\documentclass[pdflatex]{article}

\usepackage{enumerate}
\usepackage{tikz}
\usepackage{subfigure}
\usepackage[english]{babel}
\usepackage{graphicx}
\usepackage{cite}

\usepackage{amsfonts,amssymb,amsmath,latexsym,amsthm}
\usepackage{multirow}

\newtheorem{theorem}{Theorem}[section]

\newtheorem{lemma}[theorem]{Lemma}
\newtheorem{corollary}[theorem]{Corollary}
\newtheorem{conjecture}[theorem]{Conjecture}
\newtheorem{claim}{Claim}
\newtheorem{problem}[theorem]{Problem}

\theoremstyle{definition}

\begin{document}
	
\title{A note on shortest circuit cover of 3-edge colorable cubic signed graphs\thanks{The work was supported by NNSF of China (No. 12071453) and  Anhui Initiative in Quantum Information Technologies (AHY150200) and the National Key R and D Program of China(2020YFA0713100).}}
\author{Ronggui Xu$^a$, \quad Jiaao Li$^b$ \quad Xinmin Hou$^{a,c}$\\
\small $^a$School of Mathematical Sciences\\
\small University of Science and Technology of China, Hefei, Anhui 230026, China.\\
\small $^b$School of Mathematical Sciences\\
\small Nankai University, Tianjin 300071, China\\
\small $^{c}$CAS Key Laboratory of Wu Wen-Tsun Mathematics\\
\small University of Science and Technology of China, Hefei, Anhui 230026, China.
}
\date{}
\maketitle

\begin{abstract}
A {sign-circuit cover} $\mathcal{F}$ of a signed graph $(G, \sigma)$ is a family of sign-circuits which covers all edges of $(G, \sigma)$. The shortest sign-circuit cover problem was initiated by M\'a$\check{\text{c}}$ajov\'a, Raspaud, Rollov\'a, and \v{S}koviera (JGT 2016) and received many attentions in recent years.  In this paper, we show that every flow-admissible 3-edge colorable cubic signed graph $(G, \sigma)$ has a sign-circuit cover with length at most $\frac{20}{9} |E(G)|$.
	
\end{abstract}


\section{Introduction}
In this paper,  graphs may have parallel edges and loops. A {\it circuit} is a connected $2$-regular graph. A  graph is {\it even} if every vertex has even degree, and
an Eulerian graph is a connected even graph. A {\it circuit cover} $\mathcal{C}$ of a graph is a family of circuits which covers all edges of $G$. We call $\mathcal{C}$ a {\it circuit $k$-cover} of $G$ if $\mathcal{C}$ covers every edge of $G$ exactly $k$ times. The {\it length} of a circuit cover $\mathcal{C}$ is defined as $\ell(\mathcal{C})=\sum_{C\in \mathcal{C}}|E(C)|$.
Determining the shortest length of a circuit cover of a graph $G$ (denoted by $scc(G)=\min \{\ell(\mathcal{C}): \mathcal{C}~\text{is a circuit cover}\}$) is a classic optimization problem initiated by Itai, Lipton, Papadimitriou, and Rodeh \cite{1981Covering}. Thomassen~\cite{1997complexity} showed that it is NP-complete to determine
whether a bridgeless graph has a circuit cover with length at most $k$ for a given integer $k$. A well-known conjecture, the Shortest Circuit Cover Conjecture, was proposed by Alon and Tarsi~\cite{1985covering} as follows.
\begin{conjecture}[Shortest Circuit Cover Conjecture]\label{CONJ: c1}
 For any $2$-edge-connected graph $G$, $scc(G) \leq \frac{7}{5}|E(G)|$.
\end{conjecture}
The upper bound is achieved by the Petersen graph.
Jamshy and Tarsi~\cite{1992covers} proved that Conjecture~\ref{CONJ: c1} implies the well-known Cycle Double Cover Conjecture proposed by Seymour~\cite{1979sums} and Szekeres~\cite{1973decomp}. The best known general result about Conjecture~\ref{CONJ: c1} is  obtained by Bermond, Jackson, Jaeger~\cite{1983covering} and Alon, Tarsi~\cite{1985covering}, independently.
\begin{theorem}[Bermond, Jackson and Jaeger~\cite{1983covering}, Alon and Tarsi~\cite{1985covering}]\label{THM: grneral-upper}
Let $G$ be a $2$-edge-connected graph. Then $scc(G) \leq \frac{5}{3}|E(G)|$.
\end{theorem}

Several improvements of this upper bound for cubic graphs $G$ have been made in literature. Specifically, Jackson~\cite{Jack94} showed
that $scc(G)\le\frac{64}{39}|E(G)|$ and Fan~\cite{Fan94} later showed that $scc(G)\le \frac{44}{27} |E(G)|$, Kaiser, Tom{\'a}{\v{s}}, Kr{\'a}l, Lidick{\`y}, Nejedl{\`y}, and {\v{S}}{\'a}mal~\cite{Kaiser} improved Fan's result to $scc(G)\le\frac{34}{21} |E(G)|$ and Hou and Zhang \cite{2012HouZhang} proved that $scc(G)\le\frac 85 |E(G)|$ if $G$ has girth at least 7 and $scc(G)\le\frac{361}{225}|E(G)|$ if all 5-circuits of $G$ are disjoint. Recently, Luko\v{t}ka \cite{2020Lukotka} showed that $scc(G)\le\frac{212}{135} |E(G)|$ for all $2$-edge-connected cubic graphs $G$.

A {\it signed} graph $(G, \sigma)$ is a graph $G$ associated with a mapping $\sigma : E(G)\rightarrow   \{+1, -1\}$. An edge $e\in E(G)$ is {\it positive} if $\sigma(e)=1$ and {\it negative} if $\sigma(e)=-1$. A signed graph $G$ is called {\it positive} if $G$ contains even number of negative edges and otherwise called {\it negative}. In a signed graph, a circuit with an even number of negative edges is called a {\it balanced circuit}, and otherwise we call it an {\it unbalanced circuit}. A {\it barbell} is a signed graph consisting of two unbalanced circuits joined by a (possibly trivial) path, intersecting with the circuits only at ends. If the path in a barbell is trivial, the barbell is called a {\it short barbell}; otherwise, it is a {\it long barbell}.
A balanced circuit or a barbell is called a {\it sign-circuit} of a signed graph. A {\it sign-circuit cover} $\mathcal{F}$ of a signed graph is a family of sign-circuits which covers all edges of $(G, \sigma)$. In fact, it is well-known that a signed graph has a sign-circuit cover if and only if each edge lies in a sign-circuit, which is equivalent to the fact that the signed graph admits a nowhere-zero integer flow, so-called, {\it flow-admissible}. (Readers may refer to~\cite{1983Nowhere} for details).
The shortest length of a sign-circuit cover of a signed graph $(G, \sigma)$ is also denoted by $scc(G)$. The shortest sign-circuit cover problem was initiated by M\'a$\check{\text{c}}$ajov\'a, Raspaud, Rollov\'a, and \v{S}koviera~\cite{2016Circuit} and received many attentions in recent years.  It is a major open problem for the optimal upper bound of shortest sign-circuit cover in signed graphs.
\begin{problem}\label{problem5/3}
  What is the optimal constant $c$ such that $scc(G)\le c\cdot|E(G)|$ for every flow-admissible signed graph $(G,\sigma)$ ?
\end{problem}

 As remarked in \cite{2016Circuit}, the signed Petersen graph $(P,\sigma)$ whose negative edges induce a circuit of length five has $scc(P)=\frac{5}{3}|E(P)|$, which indicates $c\ge \frac{5}{3}$.  We list some of known results related to Problem~\ref{problem5/3}.

\begin{itemize}
  \item[(1)] $c\le 11$ by M\'a$\check{\text{c}}$ajov\'a, Raspaud, Rollov\'a, and \v{S}koviera~\cite{2016Circuit}.
  \item[(2)] $c\le \frac{14}{3}$ by Lu, Cheng, Luo, and Zhang~\cite{Cheng19}.
  \item[(3)] $c\le \frac{25}{6}$ by Chen and Fan~\cite{Chen18}.
  \item[(4)] $c\le \frac{11}{3}$  by Kaiser, Luko\v{t}ka, M\'a$\check{\text{c}}$ajov\'a, and Rollov\'a~\cite{Kaiser19}.
  \item[(5)] $c\le \frac{19}{6}$ by Wu and Ye~\cite{2018Minimum}.
\end{itemize}

For any flow-admissible $2$-edge-connected cubic signed graph $(G,\sigma)$, Wu and Ye \cite{2018Circuit} obtained a better upper bound that $scc(G)\le \frac{26}9|E(G)|$.
 In this article, we focus on the shortest sign-circuit cover of
3-edge colorable cubic signed graphs and prove the following theorem.
\begin{theorem}\label{THM: main}
	Every flow-admissible $3$-edge colorable cubic signed
	 graph $(G,\sigma)$ has a sign-circuit cover with length at most $\frac{20}{9} |E(G)|$.
\end{theorem}
An equivalent version of the Four-Color Theorem states that every $2$-edge-connected cubic planar graph is $3$-edge colorable.
So we have the following corollary.
\begin{corollary}
	Every flow-admissible  $2$-edge-connected cubic planar signed graph $(G,\sigma)$
	has a sign-circuit cover with length at most $\frac{20}{9} |E(G)|$.
\end{corollary}

Now we introduce more notation and terminologies used in the following sections.
Let $G$ be a graph and  $T\subseteq V(G)$ with $|T|\equiv0\pmod2$. A {\it $T$-join} $J$ of $G$ with respect to $T$ is a subset of edges of $G$ such that $d_{J}(v)\equiv 1 \pmod 2$ if and only if $v \in T$, where $d_J(v)$ denotes the degree of $v$ in the edge-induced subgraph $G[J]$. A $T$-join is minimum if it has minimum number of edges among all $T$-joins.
Let  $G^{\prime}$ be the graph obtained from a graph $G$ by  deleting  all  the  bridges  of $G$. Then the components of
$G^{\prime}$ are called the {\it bridgeless-blocks} of $G$. By the definition, a bridgeless-block is either a single vertex or a maximal $2$-edge-connected subgraph of $G$.
For a vertex subset $U \subseteq V (G)$, $\delta_{G}(U)$ denotes
the set of edges with one end in $U$ and the other in $V(G)\setminus U$.
Let $u$, $v$ be two  vertices in $V(G)$.  A $(u,v)$-path is a path connecting $u$ and $v$.
Let $C=v_{1}\ldots v_{r}v_{1}$ be a circuit where $v_{1}, v_{2},\ldots, v_{r}$ appear in clockwise on $C$.
A {\it segment} of $C$ is the path $v_{i}v_{i+1}\ldots v_{j-1}v_{j}$ (where the sum of the index is under modulo $r$) contained in $C$ and is denoted by $v_{i}Cv_{j}$.
A connected graph $H$ is called a {\it cycle-tree}~\cite{2018Circuit} if it has no vertices of degree 1 and all circuits of $H$ are edge-disjoint.
In a signed graph, {\it switching} a vertex $u$ means reversing the signs of all edges incident with
$u$. Two signed graphs are {\it equivalent} if one can be obtained from the other by a sequence of switching operations, and a signed graph is {\it balanced} if and only if it is equivalent to an ordinary graph. The set of negative edges of $(G, \sigma)$ is denoted by $E_N(G, \sigma)$.

The rest of the article is organized as follows.
Some basic lemmas about signed graph and $T$-join and a crucial lemma which deals with a special case in our proof are given in Section 2.
 Then we are able to complete the proof of Theorem \ref{THM: main} in Section 3 and we will finish with some discussions and remark.

\section{Some Lemmas}
The following lemma due to Bouchet~\cite{1983Nowhere}
characterized connected flow-admissible signed graphs.
\begin{lemma}[Bouchet \cite{1983Nowhere}]\label{LEM: flowadmissible}
 A connected signed graph $(G, \sigma)$ is flow-admissible if and only if it is not equivalent to a signed graph with exactly one negative edge and it has no bridge
$e$ such that $(G-e, \sigma |_{G-e})$ has a balanced component.
\end{lemma}

\begin{lemma}[Li, Li, Luo, Zhang and Zhang \cite{2021Cubic}]\label{LEM: spanningtree}
	Let $T $ be a spanning tree of a signed graph $G$. For every $e \notin E(T)$, let $C_e$ be the unique circuit contained in $T + e$.
	 If the circuit $C_e$  is balanced for every $e\notin E(T)$, then $G$ is balanced.
   \end{lemma}

Wu and Ye~\cite{2018Minimum} gave a lemma to control the size of a $T$-join.
\begin{lemma}[Wu and Ye \cite{2018Minimum}]\label{LEM: T-join}
Let $G$ be a 2-edge-connected graph and $T$ be subset of vertices with $|T|$ even. Then $G$ has a
	$T$-join of size at most $\frac{1}{2}|E(G)|$.
\end{lemma}

The following two results gave upper bounds of $scc(G)$ with $G$ under some constrains.
\begin{lemma}[Chen, Fan \cite{2021Circuit} and Kaiser, Luko\v{t}ka, M\'a$\check{\text{c}}$ajov\'a, Rollov\'a~\cite{Kaiser19}]\label{LEM: scc1}
 Let $(G,\sigma)$ be a signed graph and suppose that each bridgeless-block of $G$ is Eulerian.

(a)(Corollary 1.5 in \cite{2021Circuit}) If $(G,\sigma)$ is flow-admissible, then $scc(G) \leq \frac{3}{2}|E(G)|$.

(b)(Corollary 2.6 in \cite{Kaiser19}) If the union of all the bridgeless-block of $G$, denoted by $H$, is positive, then there exists a family of sign-circuits ${\mathcal{F}}$ covers $H$ with length at most $\frac{4}{3}|E(G)|$.
\end{lemma}

A combination of the above two lemmas leads to the following.
\begin{lemma}\label{10-9subcase1.1}
  Let $F$ be a $2$-factor of a $2$-edge-connected cubic sign graph $(G,\sigma)$. If $F$ contains even number of negative edges, then there exists a family of sign-circuits ${\mathcal{F}}$ covers $F$ with length at most $\frac{10}{9}|E(G)|$.
\end{lemma}

\begin{proof} We may assume that the $2$-factor $F$ consists of circuits $C_1, C_2,\ldots, C_{t}$.
  Denote by $G^{*}$ the graph obtained from $G$ by contracting each circuit $C_{i}$ of $F$ to a single vertex $c_{i}$.

  Since $F$ contains even number of negative edges, the number of unbalanced circuits in $F$ is even.
Without loss of generality, we may assume that $\mathcal{C}=\{C_{1},C_{2},...,C_{2s}\}$ is the set of unbalanced circuits of $F$. Let $T=\{c_{1},c_{2},...,c_{2s}\}$ and $J$ be a minimum $T$-join of $G^{*}$ with respect to $T$.
		Since $G$ is 2-edge-connected and $G^{*}$ is obtained from $G$ by contracting edges,  $G^{*}$ is 2-edge-connected as well. By Lemma~\ref{LEM: T-join}, we have $|J| \leq \frac{1}{2}|E(G^{*})|=\frac{1}{6}|E(G)|$. Consider the edge set $F \cup J$ in $G$, and we view it as an edge-induced subgraph of $G$. By the definition of $T$-join, we can apply Lemma~\ref{LEM: scc1}(b) to $F \cup J$, i.e., there exists a family of sign-circuits ${\mathcal{F}}$ covers $F$  with length
		\begin{align*}
			\ell({\mathcal{F}})&\leq \frac{4}{3}|E(F\cup J)|\\
			                    &= \frac{4}{3}\left(|E(F)|+|E(J)|\right)\\
								&\leq \frac{4}{3}\left(\frac{2}{3}|E(G)|+\frac{1}{6}|E(G)|\right)\\
                                & =\frac{10}{9}|E(G)|.
		\end{align*}	
This proves the lemma.
\end{proof}

The proof of the following lemma is inspired by the proof of Lemma 3.7 in \cite{2021Cubic} on flows of $3$-edge colorable cubic signed graphs.
\begin{lemma}\label{LEM: tech1}
	Let $C$ be an unbalanced circuit of a cubic signed graph $(G,\sigma)$. If $(G,\sigma)$ is flow-admissible and $G-E(C)$ is balanced, then $(G,\sigma)$ has a family $\mathcal{F}$ of sign-circuits such that

(1) $E(C)$ is covered by $\mathcal{F}$, and

(2) the length of $\mathcal{F}$ satisfies $\ell(\mathcal{F}) \leq \frac{8}{9}|E(G)|+|E(C)|$.
\end{lemma}

\begin{proof}
Let $G^{\prime}=G-E(C)$. Since $G^{\prime}$ is balanced, with some switching operations,  we may assume that all edges in $E(G^{\prime})$ are positive and thus $E_{N}(G,\sigma) \subseteq E(C)$.

	Let $M$ be a component of $G^{\prime}$.  The circuit $C$ was divided by the vertices of $M$ into
	 pairwise edge-disjoint paths (called segments) whose end-vertices lie in $M$ and all inner vertices lie
	in $C$. An end-vertex of a segment is called an attachment of $M$. A segment is called positive (negative, resp.) if it contains an even (odd, resp.) number of negative edges. Note that $M \cup S$
	is unbalanced (balanced, resp.) if and only if the segment $S$ is negative (positive, resp.). Since $M\cup C$ is unbalanced, the number of negative segments determined by $M$ is odd.

\vspace{5pt}
\noindent{\textbf{Case 1.}} {\em There exists a  component $M$ of  $G^{\prime}$ that determines more than one negative segments.}
	
Then in this case $M$ determines at least three  negative segments and so $|E(C)|\geq 3$.
	    Let $u_{1}Cv_{1}, u_{2}Cv_{2}, u_{3}Cv_{3}$ be three consecutive negative segments (in clockwise order) where $u_{i}$ and $v_{i}$ are attachments for $i=1, 2,3$. Then $v_{1}Cu_{2}, v_{2}Cu_{3}, v_{3}Cu_{1}$, where each of them contains even number of negative edges. This implies that $C$ can be partitioned into three pieces: $u_{1}Cu_{2}, u_{2}Cu_{3}$, and $u_{3}Cu_{1}$ all contain odd number
	   of negative edges. Note that $u_{i}$ and $u_{j}$ are not adjacent in $C$ for distinct $i,j \in \{ 1,2,3\}$ since $G$ is cubic.
	   Let $P_{1}$ be a $(u_{1},u_{2})$-path in $M$. Since $M$ is connected, there is a path $P_{2}$ from $u_{3}$ to $P_{1}$ such that $|V (P_{2}) \cap V (P_{1})| = 1$. Let $v$ be the only common vertex in $P_{1}$ and $P_{2}$.
	   Then $C, P_{1}, P_{2}$ form a signed graph $H_{1}$ as illustrated in Figure 1.
	   \begin{figure}
		\centering
		\includegraphics[scale=0.28]{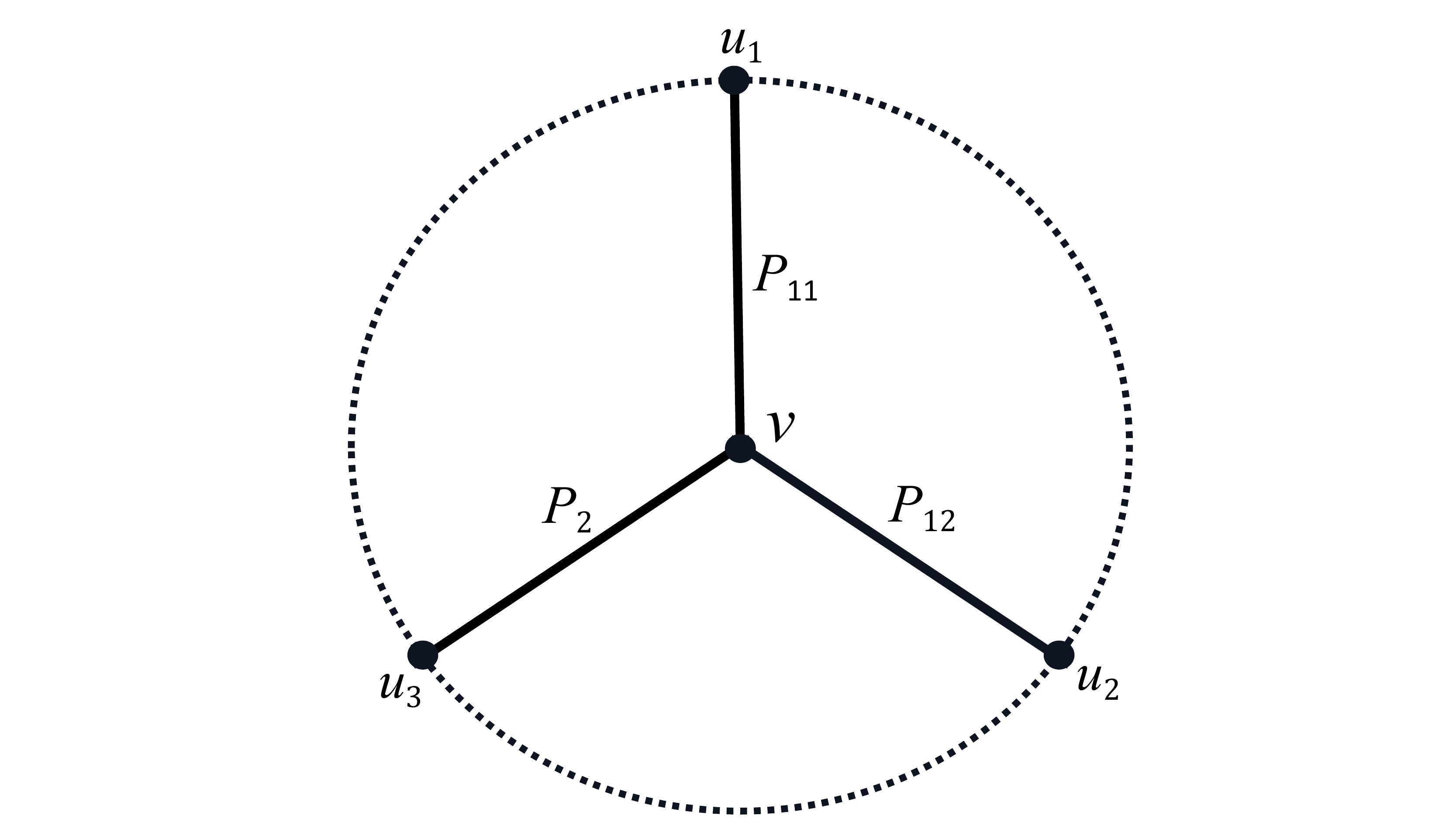}
		\caption{$H_{1}= C\cup P_{11}\cup P_{12}\cup P_{2} $, negative segments are dashed.}
		\label{Figure 1}
	   \end{figure}

Note that $|E(H_{1})|\leq \frac{2}{3}|E(G)|+2$ since $G$ is cubic and  there are exactly four vertices of degree $3$ in $H_{1}$. Divide $P_{1}$ into two pieces in $M$: $u_{1}P_{11}v$ and $vP_{12}u_{2}$, denote by $B_{1}=u_{1}P_{11}v\cup vP_{12}u_{2}\cup u_{2}Cu_{1},B_{2}=u_{3}P_{2}v\cup vP_{11}u_{1}\cup u_{1}Cu_{3},B_{3}=u_{2}P_{12}v\cup vP_{2}u_{3}\cup u_{3}Cu_{2}$.
Note that each  of $u_{2}Cu_{1}, u_{1}Cu_{3}, u_{3}Cu_{2}$ contains even number of negative edges.
So $B_{1},B_{2},B_{3}$ are all balanced circuits and $\mathcal{F}_{1}=\{B_{1},B_{2}\}$, $\mathcal{F}_{2}=\{B_{2},B_{3}\}$, $\mathcal{F}_{3}=\{B_{3},B_{1}\}$ are all sign-circuit covers of $H_{1}$, which are also sign-circuit covers of $C$ since $E(C)\subset E(H_{1})$. Note that $\mathcal{F}=\{\mathcal{F}_{1},\mathcal{F}_{2},\mathcal{F}_{3}\}$ covers edge edge of $H_{1}$ exactly $4$ times. So we have
		\begin{align*}
			\min\{ \ell(\mathcal{F}_{1}), \ell(\mathcal{F}_{2}), \ell(\mathcal{F}_{3})\}&\leq \frac{1}{3}\left(\ell(\mathcal{F}_{1})+\ell(\mathcal{F}_{2})+\ell(\mathcal{F}_{3})\right)\\
			    &=\frac{4}{3}|E(H_{1})|\\
				&\leq \frac{4}{3}\left(\frac{2}{3}|E(G)|+2\right)\\
				&< \frac{8}{9}|E(G)|+|E(C)|.
		\end{align*}

~

\noindent{\textbf{Case 2.}} {\em Each component of  $G^{\prime}$ determines exactly one negative segment.}

Let $\mathcal{M}$ denote the  set of all components  of  $G^{\prime}$. For each component $M$, denote by
	    $S_{M} = uCv$ the negative segment determined by $M$ where $u$ and $v$ are two attachments of $M$ on $C$. Denote by $S'_{M}= vCu$ the cosegment of $S_{M}$, which is the complement of $S_M$ on $C$. Then $E(S_{M})\neq \emptyset $ and
	    $S'_{M}= E(C)-E(S_{M})$. We have the following two conclusions:
\begin{claim}[see Claim 3.7.2 in~\cite{2021Cubic}]\label{CLM: cl1}
 $$\cap_{M\in \mathcal{M}}E(S_{M})=\emptyset,$$ or equivalently, $\cup _{M\in \mathcal{M}}E(S'_{M})=C$ and $|\mathcal{M}|\geq 2$.
\end{claim}
		Let $\mathcal{S}=\{S^{\prime}_{1},S^{\prime}_{2},...,S^{\prime}_{t}\}$ be a minimal cosegment cover of $C$. We have
		\begin{claim}[see Claim 3.7.3 in~\cite{2021Cubic}]\label{CLM: cl2}
 For any edge $e\in E(C)$, $e$ is contained in at most two cosegments.
		\end{claim}
\begin{proof}[Proof Sketches of Claims 1 and 2:] For the sake of completeness, we present the proof sketches of this two claims here. Suppose to the contrary  $\cap_{M\in \mathcal{M}}E(S_{M})\ne \emptyset$ and $e^{*}\in \cap_{M\in \mathcal{M}}E(S_{M}).$
 Then there is a spanning tree $T$ of $G-e^*$ containing the path $P^* = C -e^*$. Let $e = uv \in E(G)-e^*-E(T )$.
 Denote the unique circuit contained in $T + e$ by $C_e$.

 If $E(C_e)\cap E(P^*) = \emptyset$, then $C_e$ contains no negative edges and thus is balanced. Otherwise since $T$ contains all the edges in $C-e^*$,
 $E(C_e) \cap E(C)$ is a path $P$ on $C$. let $u^{\prime}$ and $v^{\prime}$ be the two end-vertices of $P$ in clockwise order on $C.$ Then $C_e-V (P) + {u^{\prime}, v^{\prime}}$ is a also a path and thus it is contained in some component $M \in \mathcal{M}.$
 This implies that $u^{\prime}$  and $v^{\prime}$ are two attachments of $M$ on $C.$  Since $e^*$ belongs to the only negative segment of $C$  determined by $M$, ${u^{\prime}Cv^{\prime}}$  is the union of some
 positive segments of $C$  determined by  $M$.  Therefore $C_e$ has an even number of negative
 edges and thus is balanced. By  Lemma~\ref{LEM: spanningtree}, $G-e^*$ is balanced, contradicting  Lemma~\ref{LEM: flowadmissible}.
 This proves $\cap_{M\in \mathcal{M}}E(S_{M})= \emptyset$.
 Since $E(S^{\prime}_M)=E(C)-E(S_M)$ and $\cap_{M\in \mathcal{M}}E(S_{M})= \emptyset$, we have $\cup_{M\in \mathcal{M}}E(S^{\prime}_M)= C$.
 Since $E(S_M)\ne \emptyset$ and $\cap_{M\in \mathcal{M}}E(S_{M})= \emptyset$, we have $|\mathcal{M}|\geq 2$. This completes the proof of the claim 1.

 Suppose to the contrary that there exists an edge $e = uv$ that belongs to three cosegments $L_1,L_2,L_3$ of $\mathcal{S}$. Denote $L_i=u_iCv_i$ for each $i=1,2,3$. Without loss
 of generality, we may assume that $u_2$ belongs to $u_1Cu$. Then $v_2$ doesn’t belong to $u_1Cv_1$ (see Figure 2). Note that $v_3$ belongs to $u_1Cu_3$.
 If $u_3$ belongs to $u_1Cu$, then both $v_3$ and $u_3$ belongs to $u_1Cv_1$ and thus $L_1 \cup L_3 = C$ (see Figure 2-(a)), contradicting the minimality of $\mathcal{S}$.
 If $u_3$ doesn’t belong to $u_1Cu$, then $u_3$ belongs to $vCv_2$. Since $L_3$ contains $uv$, $v_3$ belongs to $vCv_2$. Thus both $v_3$ and $u_3$ belongs to $u_2Cv_2$. Therefore $L_2 \cup L_3 = C$ (see Figure 2-(b)), also contradicting the minimality of $\mathcal{S}$. This completes the proof of the claim 2.
\end{proof}

 \begin{figure}
	\centering
   \includegraphics[scale=0.28]{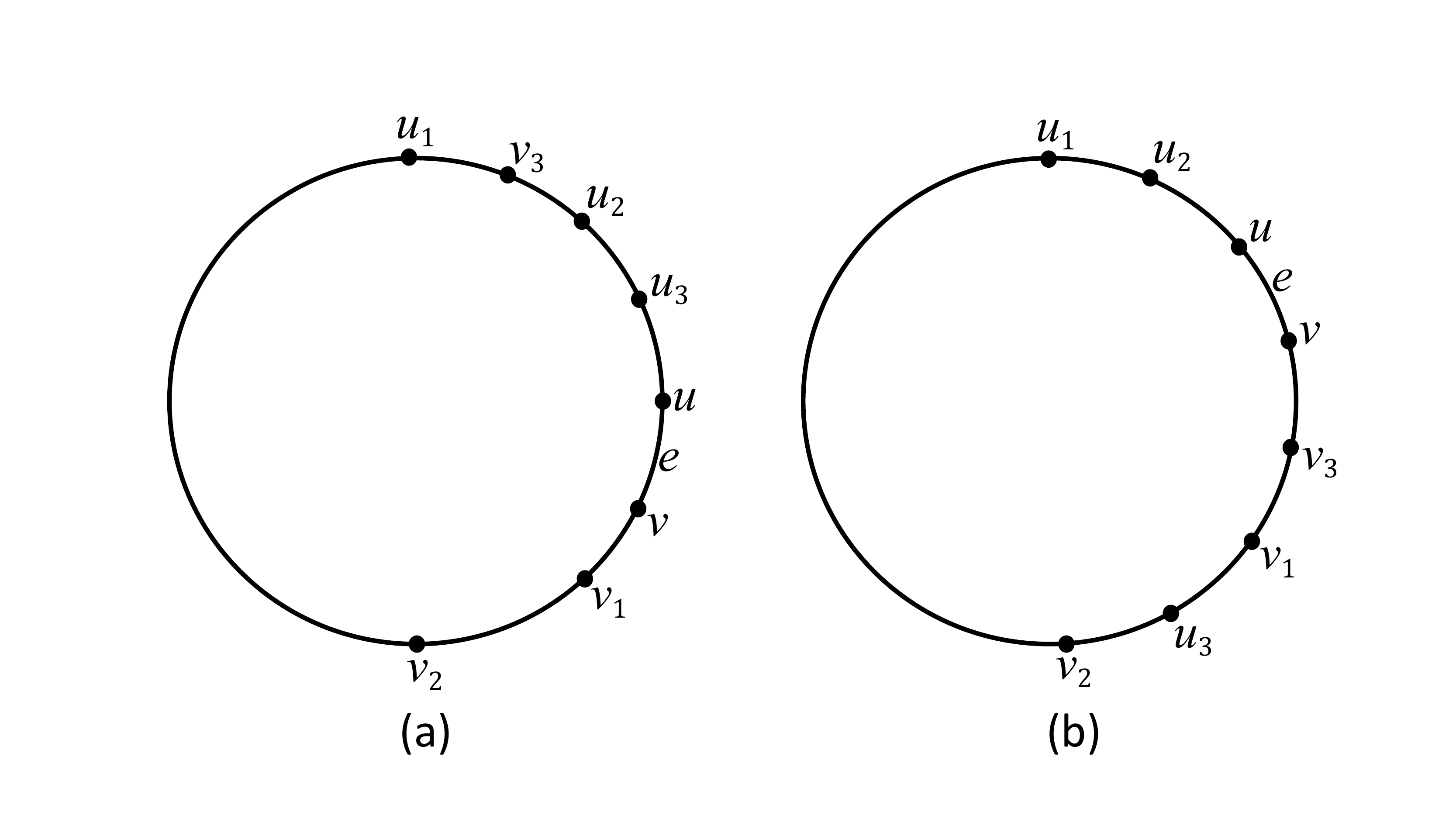}
   \caption{Three cosegments sharing an edge.}
   \label{Figure 2}
\end{figure}

		For each $i= 1,...,t$, denote by $S^{\prime}_{i} = x_{i}Cy_{i}$ and let $P_{i}$ be a path in $M_{i}$ connecting $x_{i}$ and $y_{i}$. Then $C_{i}=S^{\prime}_{i} \cup P_{i}$ is a balanced Eulerian subgraph. By Claims~\ref{CLM: cl1} and \ref{CLM: cl2}, we may assume that the vertices $x_{1}, y_{t}, x_{2}, y_{1},..., x_{t}, y_{t-1}, x_{1}$ appear on $C$ in clockwise order.  Then $C_{i}\cap C_{j} \neq \emptyset$  if and only if $|j - i| \equiv 1 \pmod t$. Moreover $|C_{i} \cap C_{i+1}| = x_{i+1}Cy_{i}$,
  where the sum of the indices are taken modulo $t$. See Figure 3 for an illustration with $t = 5$.
\begin{figure}
 	\centering
	\includegraphics[scale=0.28]{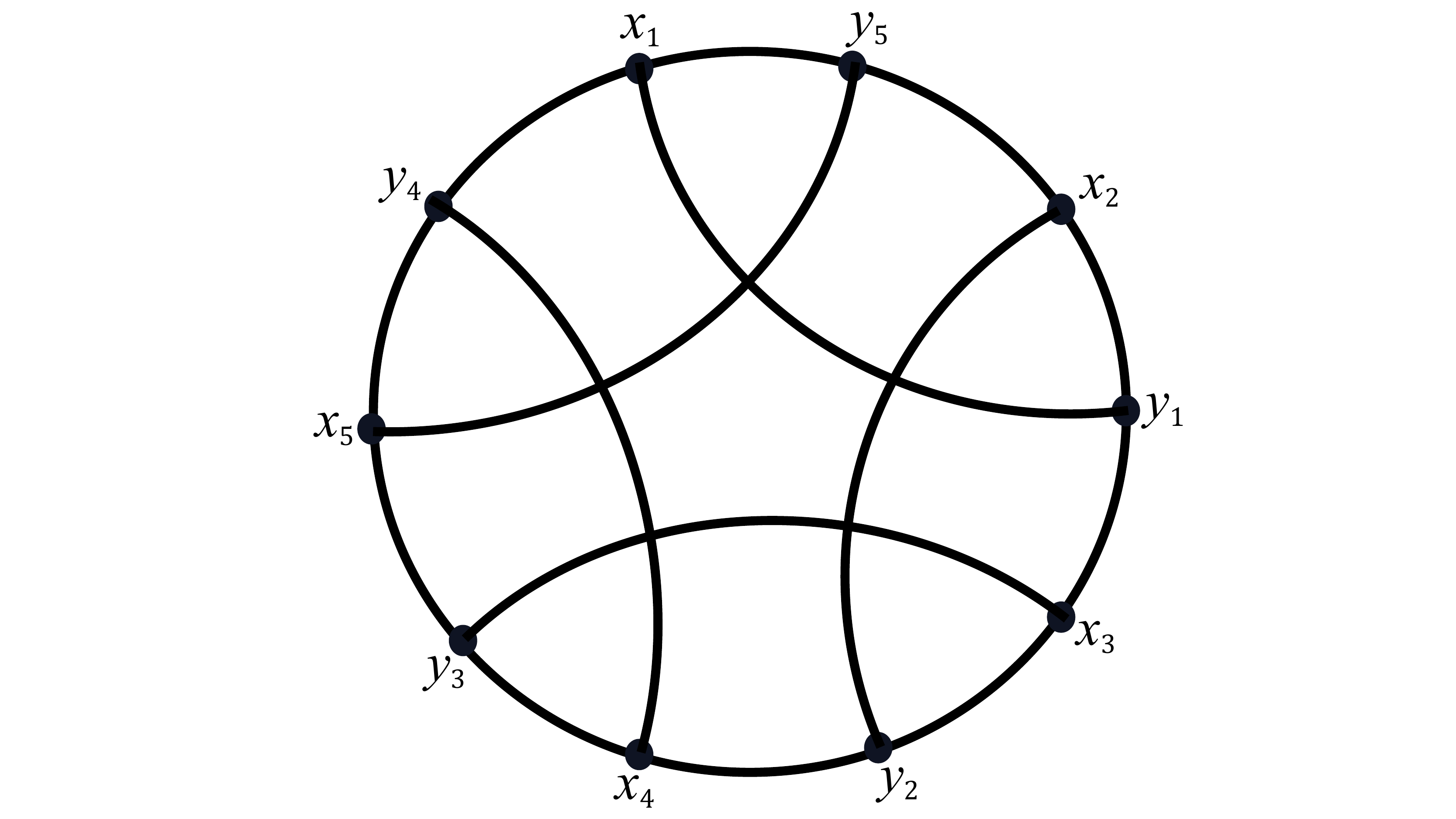}
	\caption{Minimal cosegment cover of $C$ with $t=5$.}
	\label{Figure 3}
\end{figure}

Let $H_{2}=C \cup P_{1} \cup P_{2} \cup ...\cup P_{t}$ and $B_{i}=x_{i}Cy_{i}\cup y_{i}P_{i}x_{i}$ for $i=1,2,...,t$. Note that $B_{i}$ is a balanced circuit and so $\mathcal{F}=\{B_{1},B_{2},...,B_{t}\}$ is a sign-circuit cover of $C$.
	Obviously, $|E(H_{2})|\leq \frac{2}{3}|E(G)|+t$ since $G$ is cubic and  there are exactly $2t$ vertices of degree $3$ in $H_{2}$. By Claim~\ref{CLM: cl2}, $\mathcal{F}$ covers the edges in $C$ at most twice and edges in $P_{1} \cup \ldots\cup P_{t}$ exactly once. Let $W$ be the set of edges covered by $\mathcal{F}$ twice. Then $W\subset E(C)$. Since $G$ is cubic, $x_{i}\neq y_{j}$ for all $i,j\in\{1,2,...,t\}$.  So we have $|W|\leq |E(C)|-t$.
Therefore,
	\begin{align*}
	   \ell(\mathcal{F})&=|E(B_{1})|+|E(B_{2})|+...+|E(B_{t})|\\
	                 &\leq |E(H_{2})|+|E(C)|-t\\
					 &\leq \frac{2}{3}|E(G)|+t+|E(C)|-t\\
					 &\leq \frac{8}{9}|E(G)|+|E(C)|.
	\end{align*}

This completes the proof.
\end{proof}

\section{Proof of Theorem~\ref{THM: main}}
\begin{proof}[Proof of Theorem~\ref{THM: main}]
Let $f$ be a 3-edge coloring of connected cubic graph $G$.
Let $R, B, Y$ be the three color classes of $f$. Recall that $E_{N}(G, \sigma)$ is the set of negative edges in $(G, \sigma)$. Without loss of generality, we may assume $\left|R \cap E_{N}(G, \sigma)\right| \equiv\left|B \cap E_{N}(G, \sigma)\right|\pmod 2.$
	Denote by $M_{1}M_{2}$ the $2$-factor induced by $M_{1}\cup M_{2}$ for each pair $M_{1}, M_{2} \in \{R, B, Y \}$.
	Since $|R \cap E_{N}(G, \sigma)| \equiv |B \cap E_{N}(G, \sigma)| \pmod 2$, $RB$ has an even number of unbalanced components. By Lemma~\ref{10-9subcase1.1}, there exists a family of sign-circuits ${\mathcal{F}}_{1}$ covers $RB$ with length at most $\frac{10}{9}|E(G)|$.

\vspace{5pt}
\noindent{\textbf{Case 1.}} {\em $RB$ contains an unbalanced circuit.}

	First, assume that $|Y \cap E_{N}(G, \sigma)|$ has the same parity with $|R \cap E_{N}(G, \sigma)|$.
Then $RY$ has an even number of unbalanced circuits. By Lemma~\ref{10-9subcase1.1}, we can find a family of sign-circuits ${\mathcal{F}}_{2}$ covers $RY$ with length at most $\frac{10}{9}|E(G)|$.
			  Therefore, $\mathcal{F}={\mathcal{F}}_{1} \cup {\mathcal{F}}_{2}$ is a sign-circuit cover of $G$ with length
			  \begin{align*}
				\ell(\mathcal{F}) =\ell({\mathcal{F}}_{1}) +\ell({\mathcal{F}}_{2})
				               \leq \frac{10}{9}|E(G)| + \frac{10}{9}|E(G)|
							   =\frac{20}{9}|E(G)|.
			  \end{align*}
			
Then, assume instead that $|Y \cap E_{N}(G, \sigma)|$ has different parity with $|R \cap E_{N}(G, \sigma)|$.
              Let $C$ be an unbalanced circuit in $RB$.  Now we swap the colors $R$ and $B$ on $C$, i.e. reset $R^{\prime} = R \triangle E(C)$  and $B^{\prime} = B \triangle E(C)$ respectively. Note that the operation will change the parity of $|R \cap E_{N}(G, \sigma)|$ and $|B \cap E_{N}(G, \sigma)|$.  This implies that
              $|Y \cap E_{N}(G,\sigma)| \equiv |R^{\prime} \cap E_{N}(G,\sigma)| \equiv |B^{\prime} \cap E_{N}(G,\sigma)| \pmod 2$ now.  So, similar as the previous paragraph, we apply Lemma~\ref{10-9subcase1.1} to find a family of sign-circuits ${\mathcal{F}}_{2}'$ covers $R'Y$ with length at most $\frac{10}{9}|E(G)|$. Notice that ${\mathcal{F}}_{1}$ covers $RB=R'B'$ with length at most $\frac{10}{9}|E(G)|$. Hence $\mathcal{F}={\mathcal{F}}_{1} \cup {\mathcal{F}}_{2}'$ is a sign-circuit cover of $G$ with length at most $\frac{20}{9}|E(G)|$.

~

		    \item[\textbf{Case 2.}] {\em $RB$ contains no unbalanced circuit.}

			In this case, each circuit $C_{i}$ of $RB$ is a balanced circuit. Let ${\mathcal{F}}_{1}=\{ C_{i} |  C_{i}$ is a balanced circuit of $RB \}$. Then ${\mathcal{F}}_{1}$ is a family of  sign-circuits covering $RB$ with length $\ell(\mathcal{F}_1)=E(RB)=\frac{2}{3}|E(G)|$.

\medskip			
\noindent{\textbf{Subcase 2.1}:} {\em The number of unbalanced circuits in $RY$ or $BY$ is even.}
			
 By Lemma~\ref{10-9subcase1.1}, we have a family of sign-circuits ${\mathcal{F}}_{2}$ which covers $RY$ or $BY$ with length at most $\frac{10}{9}|E(G)|$. Therefore $\mathcal{F}={\mathcal{F}}_{1} \cup {\mathcal{F}}_{2}$ is a sign-circuit cover of $G$ with length $\ell(\mathcal{F})\le \frac{16}{9}|E(G)|$.

\medskip
\noindent{\textbf{Subcase 2.2}:} {\em The number of unbalanced circuits in $RY$ or $BY$ is equal to one.}

				 Without loss of generality, assume that $RY$ has exactly one unbalanced component, say, $C_{1}$. Let $\mathcal{C} = \{C_{1},..., C_{m}\}$ be the set of components of $RY$, where each $C_{i}\,  (i\geq 2)$ is balanced.
				 Let ${\mathcal{F}}_{2}=\{ C_{i} :  i\geq 2\}$. Then ${\mathcal{F}}_{2}$ is a family of sign-circuits covering $RY-E(C_{1})$ with length $\frac{2}{3}|E(G)|-|E(C_1)|$. We consider the following two cases in order to cover $C_{1}$.

Assume first that $G$ contains an unbalanced circuit $C^{\prime}$ with $E(C')\cap E(C_{1})=\emptyset$.				
					Since $G$ is cubic and connected, there is a long barbell $Q$ in $G$
                    with $P$ as the path connecting $C_{1}$ and $C^{\prime}$ with $|E(Q)\leq \frac{2}{3}|E(G)|+1$.
                   Therefore, $\mathcal{F}={\mathcal{F}}_{1} \cup {\mathcal{F}}_{2} \cup Q$ is a sign-circuit cover of  $G$ with length
                \begin{align*}
	                \ell(\mathcal{F}) &=\ell({\mathcal{F}}_{1}) +\ell({\mathcal{F}}_{2}) +\ell(Q) \\
					               &\leq \frac{2}{3}|E(G)| + \frac{2}{3}|E(G)| -|E(C_1)| + \frac{2}{3}|E(G)| +1\\
								   &< 2|E(G)|.
				\end{align*}

Then assume instead that $G$ contains no unbalanced circuit $C^{\prime}$ with $E(C')\cap E(C_{1})=\emptyset$.
					In this case, $G-E(C_{1})$ is balanced. By Lemma~\ref{LEM: tech1}, there exists a family ${\mathcal{F}}_{3}$ of sign-circuits covering $E(C_1)$ with length at most $\frac{8}{9}|E(G)|+|E(C_1)|$. Therefore, $\mathcal{F}={\mathcal{F}}_{1} \cup {\mathcal{F}}_{2} \cup {\mathcal{F}}_{3}$ is a sign-circuit cover of $G$ with length
					\begin{align*}
						\ell(\mathcal{F}) &= \ell({\mathcal{F}}_{1}) + \ell({\mathcal{F}}_{2}) +\ell({\mathcal{F}}_{3}) \\
									   &\leq \frac{2}{3}|E(G)| + \frac{2}{3}|E(G)| -|E(C_1)| + \frac{8}{9}|E(G)| +|E(C_1)|\\
									   &\leq \frac{20}{9}|E(G)|.
					\end{align*}

\medskip
\noindent{\textbf{Subcase 2.3}:} {\em The number of unbalanced circuits in each of $RY$, $BY$  is odd and is at least $3$.}

                Let $\mathcal{C} = \{C_{1},..., C_{m}\}$ be the set of components of $RY$.
				Denote by $G^{*}$ the graph obtained from $G$ by contracting each circuit $C_{i}$ of $RY$ to a single vertex $u_{i}$, where $i=1,2,...,m$. Note that $G^{*}$ is connected, and so let $T^{*}$ be a spanning tree of $G^{*}$. Then $T^{*} \cup RY$ is a cycle-tree in $G$, denoted by $H$, containing at least $3$ unbalanced circuits.
Let $B'$ be the set of bridges of $H$ such that $b_{i} \in B'$ if and only if $(H-b_{i}, \sigma |_{H-b_{i}})$ has a balanced component $H_{i}$. $B=\emptyset$ if no such bridge exists in $H$. Let $H^{\prime}= H-(B'\cup (\cup_{i=1,...,|B'|}E(H_{i}))$. Note that $H^{\prime}$ contains all the unbalanced circuit of $\mathcal{C}$. By Lemmas~\ref{LEM: flowadmissible} and~\ref{LEM: scc1}(a), $H^{\prime}$ is flow-admissible and has a sign-circuit cover ${\mathcal{F}}_{2}$ with length at most $\frac{3}{2}|E(H^{\prime})|$.
				Since the circuits in $H_{i}$ are all balanced circuits, we can cover them with length at most $|\cup_{i=1,...,|B'|}E(H_{i})|\le |E(H)|-|E(H^{\prime})|$. Thus we have a sign-circuit cover $\mathcal{F}_{3}$  of $RY$ with length
				\begin{align*}
				  \ell(\mathcal{F}_{3}) &= \ell(\mathcal{F}_{2}) + (|E(H)|-|E(H^{\prime})|)\\
                                        &\leq \frac{3}{2}|E(H^{\prime})| + |E(H)|-|E(H^{\prime})|\\
                                     & \leq \frac{3}{2}|E(H)|
                                      < \frac{3}{2}|E(G)|.
				\end{align*}
Therefore, $\mathcal{F} = \mathcal{F}_{1} \cup \mathcal{F}_{3}$ is a sign-circuit cover of $G$ with length
				\begin{align*}
				\ell(\mathcal{F}) &= \ell(\mathcal{F}_{1}) +\ell(\mathcal{F}_{3})\\
                                  & \leq \frac{2}{3}|E(G)| +\frac{3}{2}|E(G)|\\
				                   &= \frac{13}{6}|E(G)|
                                  < \frac{20}{9}|E(G)|.
				\end{align*}
	This completes the proof.
\end{proof}
\hspace{-6.5mm}\textbf{Remark.} The upper bound of $scc(G)$ in Theorem~\ref{THM: main} seems not to be tight. We realized that the 3-edge colorable cubic signed graph $(G, \sigma)$ as illustrated in Figure 4 has a sign-circuit cover with length $\frac{13}{9}|E(G)|$. The problem to determine the optimal upper bound for the shortest sign-circuit cover of 3-edge colorable cubic signed graph remains open.

\begin{figure}
	\centering
	\includegraphics[scale=0.26]{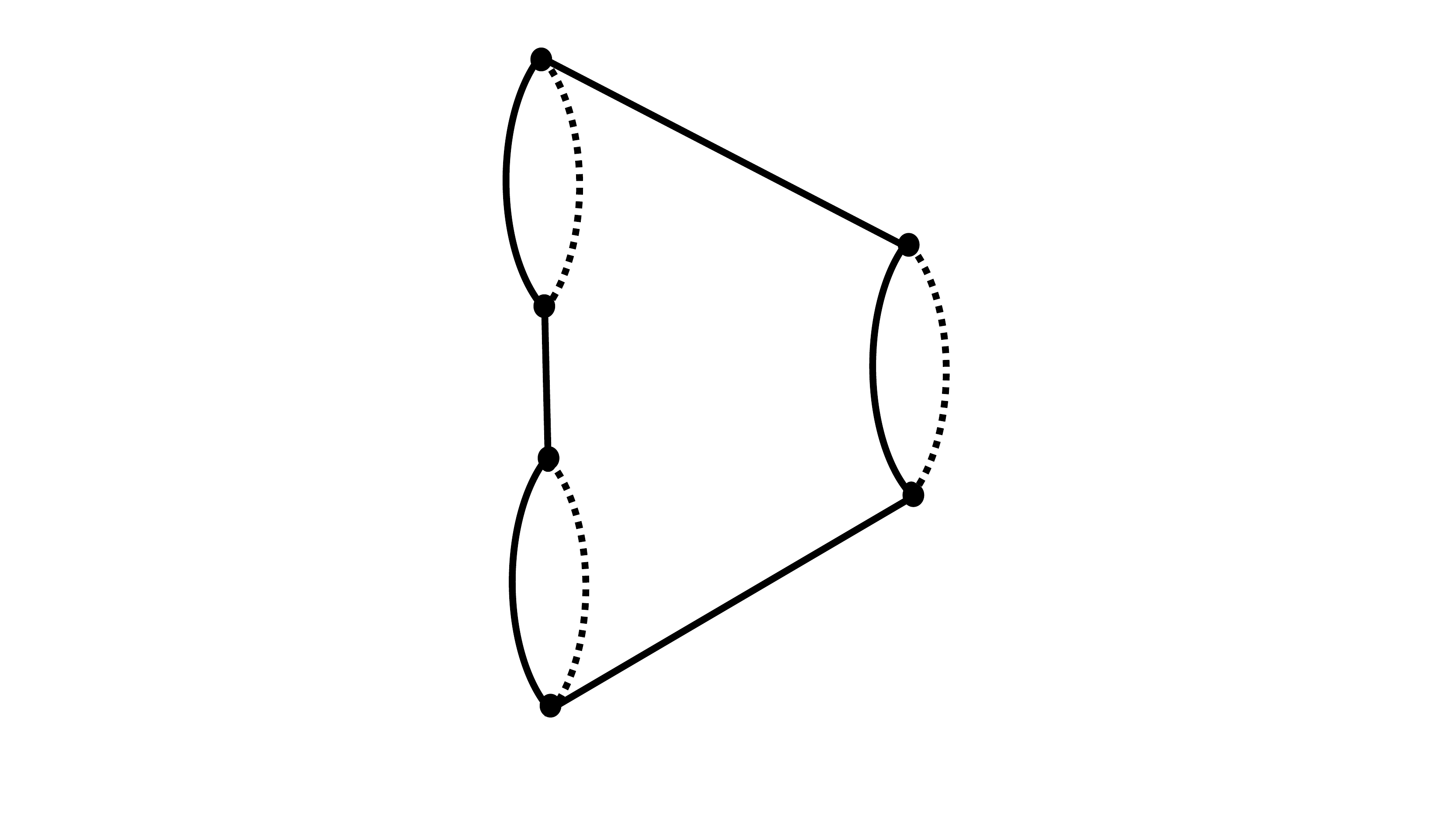}
	\caption{$(G, \sigma)$ with a shortest circuit cover with length $\frac{13}{9}|E(G)|$.}
	\label{Figure 4}
\end{figure}

\hspace{-6.5mm}\textbf{Data Availability:}
Data sharing not applicable to this article as no datasets were generated or analysed during the current study.

\end{document}